\documentclass{article}
\usepackage{amssymb}
\usepackage{amsfonts}
\usepackage{amsmath}

\newtheorem{theorem}{Theorem}[section]

\newtheorem{lemma}[theorem]{Lemma}

\newtheorem{remark}[theorem]{Remark}

\newenvironment{proof}[1][Proof]{\noindent\textbf{#1.} }{\ \rule{0.5em}{0.5em}}
\input{tcilatex}
\begin{document}

\title{Transmission Problem Between Two Herschel-Bulkley Fluids}
\date{}
\author{Farid Messelmi\thanks{%
Department of Mathematics and Computer Sciences, University of Djelfa Po Box
3117, Djelfa 17000, Algeria. }}
\maketitle

\begin{abstract}
The paper is devoted to the study of transmission probelm between two
Herschel-Bulkley fluids with different viscosities, yield limits and power
law index.


2000 Mathematics Subject Classiffication.76A05, 49J40, 76B15.
\end{abstract}

\section{Introduction}

The rigid viscoplastic and incompressible fluid\ of Herschel-Bulkley has
been scrutinized and studied by mathematicians, physicists and engineers as
intensively as the Navier-Stokes. While this model describes adequately a
large class of flows. It has been used to model the flow of metals, plastic
solids and a variety of polymers like paints. Due to existence of yield
limit, the model can capture phenomena connected with the development of
discontinuous stresses. The literature concerning this topic is extensive;
see e.g. \cite{r1}, \cite{r6}, \cite{r7}, \cite{r8} and references therein.

The purpose of this paper is to formulate and prove the existence of weak
solutions for a class of boundary transmission problems for the
Herschel-Bulkley fluid. To this aim we consider a transmission problem
between two Herschel-Bulkley fluids with different viscosities, yield limits
and power law index. We suppose that there is no-slip at the contact
interface. Such problem can model the superposition of two paints in drawing
process.

The paper is organized as follows. In Section 2 we present the steady-state
mechanical problem of transmission between two Herschel-Bulkley fluids. We
introduce some notations and preliminaries. Moreover, we derive the
variational formulation of the problem. In Section 3, we are interested in
the existence of weak solutions.

\section{Problem Statement}

We consider a mathematical problem modelling the steady-state transmission
problem between two rigid, viscoplastic and incompressible Herschel-Bulkley
fluids flow. The first fluid occupies a bounded domain $\Omega _{1}\subset 
\mathbb{R}
^{n}$ $\left( n=2,3\right) $ with the boundary $\partial \Omega _{1}$ of
class $C^{1}.$ The second one occupies a bounded domain $\Omega _{2}\subset 
\mathbb{R}
^{n}$ $\left( n=2,3\right) $ with the boundary $\partial \Omega _{2}$ of
class $C^{1}.$ We denote by $\Omega $ the domain $\Omega _{1}\cup \Omega
_{2} $ and we suppose that

\begin{equation*}
\partial \Omega _{1}=\Gamma _{0}\cup \Gamma _{1}\text{ and }\partial \Omega
_{2}=\Gamma _{0}\cup \Gamma _{2},
\end{equation*}%
where $\Gamma _{0},\Gamma _{1},\Gamma _{2}$ are measurable domains and $%
meas\left( \Gamma _{1}\right) ,meas\left( \Gamma _{2}\right) >0.$The fluids
are acted upon by given volume forces of densities $\mathbf{f}_{1},\mathbf{f}%
_{2}$ respectively.

We denote by $\mathbb{S}_{n}$ the space of symmetric tensors on $%
\mathbb{R}
^{n}.$ We define the inner product and the Euclidean norm on $%
\mathbb{R}
^{n}$ and $\mathbb{S}_{n},$ respectively, by%
\begin{equation*}
\mathbf{u}\cdot \mathbf{v}=u_{l}v_{l}\ \text{ \ }\forall \text{ }\mathbf{u},%
\text{ }\mathbf{v}\in 
\mathbb{R}
^{n}\text{ \ and}\ \ \sigma \cdot \tau =\sigma _{lm}\tau _{lm}\ \text{ \ }%
\forall \text{ }\sigma ,\text{ }\tau \in \mathbb{S}_{n}.
\end{equation*}%
\begin{equation*}
\left\vert \mathbf{u}\right\vert =\left( \mathbf{u}\cdot \mathbf{u}\right) ^{%
\frac{1}{2}}\ \ \ \forall \mathbf{u}\in 
\mathbb{R}
^{n}\text{\ \ and\ \ }\left\vert \sigma \right\vert =\left( \sigma \cdot
\sigma \right) ^{\frac{1}{2}}\ \ \ \forall \sigma \in \mathbb{S}_{n}.
\end{equation*}

Here and below,\ the indices$\ i$ and $j$\ run from\ $1$ to $n$ and
Einstein's convention is used. We denote by $\tilde{\sigma}$ the deviator of 
$\sigma =\left( \sigma _{lm}\right) $ given by%
\begin{equation*}
\tilde{\sigma}=\left( \tilde{\sigma}_{lm}\right) ,\text{ \ }\tilde{\sigma}%
_{lm}=\sigma _{lm}-\dfrac{\sigma _{ll}}{n}\delta _{lm},
\end{equation*}%
\noindent where $\delta =\left( \delta _{lm}\right) $ denotes the identity
tensor.

Let $1<p\leq 2.$ We consider the rate of deformation operator\ defined for
every $\mathbf{u}\in W^{1,p}\left( \Omega \right) ^{n}$ by%
\begin{equation*}
D\left( \mathbf{u}\right) =\left( D_{lm}\left( \mathbf{u}\right) \right) ,%
\text{ \ \ }D_{lm}\left( \mathbf{u}\right) =\dfrac{1}{2}\left(
u_{l,m}+u_{l,m}\right) .
\end{equation*}

We denote by $\mathbf{n}$ the unit outward normal vector on the boundary $%
\Gamma _{0}$ oriented to the exterior of $\Omega _{1}$ and to the interior
of $\Omega _{2},$ see the figure below$.$ For every vector field $\mathbf{v}%
\in W^{1,p}\left( \Omega _{i}\right) ^{n}$ we also write $\mathbf{v}$ for
its trace on $\partial \Omega _{i},$ $i=1,2.$

The steady-state transmission problem for the Herschel-Bulkley fluids is
given by the following mechanical problem.

\textit{Problem P1.}\ Find the velocity fields $\mathbf{u}_{i}=\left(
u_{il}\right) :\Omega _{i}\longrightarrow 
\mathbb{R}
^{n}\ $and the stress\ field $\sigma _{i}=\left( \sigma _{ilm}\right)
:\Omega _{i}\longrightarrow \mathbb{S}_{n},$ $i=1,$ $2$\ such that%
\begin{equation}
\mathbf{u}_{1}\cdot \nabla \mathbf{u}_{1}=\func{div}\sigma _{1}+\mathbf{f}%
_{1}\ \text{in\ }\Omega _{1}.  \tag{2.1}  \label{2.1}
\end{equation}%
\begin{equation}
\mathbf{u}_{2}\cdot \nabla \mathbf{u}_{2}=\func{div}\sigma _{2}+\mathbf{f}%
_{2}\ \text{in\ }\Omega _{2}.  \tag{2.2}  \label{2.2}
\end{equation}%
\begin{equation}
\left. 
\begin{array}{l}
\tilde{\sigma}_{1}=\mu _{1}\left\vert D\left( \mathbf{u}_{1}\right)
\right\vert ^{p_{1}-2}D\left( \mathbf{u}_{1}\right) +g_{1}\dfrac{D\left( 
\mathbf{u}_{1}\right) }{\left\vert D\left( \mathbf{u}_{1}\right) \right\vert 
}\text{\ \ if\ }\left\vert D\left( \mathbf{u}_{1}\right) \right\vert \neq 0
\\ 
\left\vert \tilde{\sigma}_{1}\right\vert \leq g_{1}\ \text{\ \ \ \ \ \ \ \ \
\ \ \ \ \ \ \ \ \ \ \ \ \ \ \ \ \ \ \ \ \ \ \ \ \ \ \ \ \ \ \ \ \ \ if }%
\left\vert D\left( \mathbf{u}_{1}\right) \right\vert =0%
\end{array}%
\right\} \text{\ in\ }\Omega _{1}.  \tag{2.3}  \label{2.3}
\end{equation}%
\begin{equation}
\left. 
\begin{array}{l}
\tilde{\sigma}_{2}=\mu _{2}\left\vert D\left( \mathbf{u}_{2}\right)
\right\vert ^{p_{2}-2}D\left( \mathbf{u}_{2}\right) +g_{2}\dfrac{D\left( 
\mathbf{u}_{2}\right) }{\left\vert D\left( \mathbf{u}_{2}\right) \right\vert 
}\text{\ \ if\ }\left\vert D\left( \mathbf{u}_{2}\right) \right\vert \neq 0
\\ 
\left\vert \tilde{\sigma}_{2}\right\vert \leq g_{2}\ \text{\ \ \ \ \ \ \ \ \
\ \ \ \ \ \ \ \ \ \ \ \ \ \ \ \ \ \ \ \ \ \ \ \ \ \ \ \ \ \ \ \ \ \ if }%
\left\vert D\left( \mathbf{u}_{2}\right) \right\vert =0%
\end{array}%
\right\} \text{\ in\ }\Omega _{2}.  \tag{2.4}  \label{2.4}
\end{equation}%
\begin{equation}
\func{div}\mathbf{u}_{1}=0\text{\ in\ }\Omega _{1}.  \tag{2.5}  \label{2.5}
\end{equation}%
\begin{equation}
\func{div}\mathbf{u}_{2}=0\text{\ in\ }\Omega _{2}.  \tag{2.6}  \label{2.6}
\end{equation}%
\begin{equation}
\mathbf{u}_{1}=0\text{\ on\ }\Gamma _{1}.  \tag{2.7}  \label{2.7}
\end{equation}%
\begin{equation}
\mathbf{u}_{2}=0\text{\ on\ }\Gamma _{2}.  \tag{2.8}  \label{2.8}
\end{equation}%
\begin{equation}
\mathbf{u}_{1}-\mathbf{u}_{2}=0\text{\ on\ }\Gamma _{0}.  \tag{2.9}
\label{2.9}
\end{equation}%
\begin{equation}
\sigma _{1}\cdot \mathbf{n}-\sigma _{2}\cdot \mathbf{n}=0\text{\ on\ }\Gamma
_{0}.  \tag{2.10}  \label{2.10}
\end{equation}

Here, the flows in the domains $\Omega _{1},$ $\Omega _{2}$ are given,
respectively, by equations (\ref{2.1}) and (\ref{2.2}) where the densities
is assumed equal to one for the two fluids. Equations (\ref{2.3}) and (\ref%
{2.4}) represent, respectively, the constitutive laws of the two
Herschel-Bulkley fluids where $\mu _{1},$ $\mu _{2}>0$ and $g_{1},$ $g_{2}>0$
are the consistencies and yield limits of the two fluids, respectively, $%
1<p_{1},$ $p_{2}\leq 2$ are the power law index of the two fluids,
respectively. (\ref{2.5}) and (\ref{2.6}) represent the incompressibility
conditions for the two fluids, respectively. (\ref{2.7}) and (\ref{2.8})
give the velocities on the boundaries $\Gamma _{1}$ and $\Gamma _{2},$
respectively. Finally, on the boundary part $\Gamma _{0},$\ (\ref{2.9})\ and
(\ref{2.10})\ represent the transmission condition for liquid-liquid
interface.

Existence of weak solutions for the flow of Herschel-Bulkley fluid was shown
in 1969 for $p\geq \frac{3n}{n+2}$ for which the energy equality holds and
higher differentiability techniques can be applied, in 1997 for $p\geq \frac{%
2n}{n+1},$ and recently for $p>\frac{2n}{n+2}$ using the Lipschitz
truncation method. Moreover, some existence results has been obtained for
the thermal flow in 2010 concerning the case $p\geq \frac{3n}{n+2},$ see 
\cite{r2}, \cite{r4}, \cite{r5}, \cite{r6}, \cite{r8} and \cite{r9}. Up to
now, there are only a few results concerning the regularity of weak
solutions, especially in three-dimensional domains.

We us consider the following function spaces%
\begin{equation*}
V\left( \Omega _{i}\right) =\left\{ \mathbf{v}\in W^{1,p_{i}}\left( \Omega
_{i}\right) ^{n}:\func{div}\mathbf{v}\text{ in }\Omega _{i}\text{ and }%
\mathbf{v}=0\text{ on }\Gamma _{i}\right\} ,\text{ }i=1,2.
\end{equation*}

\begin{equation*}
V=\left\{ \left( \mathbf{v}_{1},\mathbf{v}_{2}\right) \in V\left( \Omega
_{1}\right) \times V\left( \Omega _{2}\right) :\text{ }\mathbf{v}_{1}-%
\mathbf{v}_{2}=0\text{ on }\Gamma _{0}\right\} .
\end{equation*}

$V\left( \Omega _{i}\right) ,$ $i=1,2$ is a Banach space equipped with the
norm%
\begin{equation*}
\left\Vert \mathbf{v}\right\Vert _{V\left( \Omega _{i}\right) }=\left\Vert 
\mathbf{v}\right\Vert _{W^{1,p_{i}}\left( \Omega _{i}\right) ^{n}},
\end{equation*}%
and $V$ becomes a Banach space for the following norm%
\begin{equation*}
\left\Vert \left( \mathbf{v}_{1},\mathbf{v}_{2}\right) \right\Vert
_{V}=\left\Vert \mathbf{v}_{1}\right\Vert _{V\left( \Omega _{1}\right)
}+\left\Vert \mathbf{v}_{2}\right\Vert _{V\left( \Omega _{2}\right) }.
\end{equation*}

For the rest of this article, we will denote by $c$ possibly different
positive constants depending only on the data of the problem and denoting by 
$p^{\prime }$ the conjugate of $p.$

Let us introduce the functional $B_{i},$ $i=1,2$ and the operator $\phi $
defined by 
\begin{gather}
B_{i}:V\left( \Omega _{i}\right) \times V\left( \Omega _{i}\right) \times
V\left( \Omega _{i}\right) \longrightarrow \mathbf{%
\mathbb{R}
},  \notag \\
B_{i}\left( \mathbf{v}_{1},\mathbf{v}_{2},\mathbf{v}_{3}\right)
=\dint_{\Omega _{i}}\mathbf{v}_{1}\cdot \nabla \mathbf{v}_{2}\cdot \mathbf{v}%
_{3}dx,\text{ }i=1,2\mathbf{.}  \tag{2.11}  \label{2.11}
\end{gather}%
\begin{gather}
\phi :V\longrightarrow V^{\prime },\text{ }\left( \mathbf{u}_{1},\mathbf{u}%
_{2}\right) \longmapsto \phi \left( \mathbf{u}_{1},\mathbf{u}_{2}\right)
:\forall \left( \mathbf{v}_{1},\mathbf{v}_{2}\right) \in V  \notag \\
\left\langle \phi \left( \mathbf{u}_{1},\mathbf{u}_{2}\right) ,\left( 
\mathbf{v}_{1},\mathbf{v}_{2}\right) \right\rangle _{V^{\prime }\times
V}=\mu _{1}\dint_{\Omega _{1}}\left\vert D\left( \mathbf{u}_{1}\right)
\right\vert ^{p_{_{1}}-2}D\left( \mathbf{u}_{1}\right) \cdot D\left( \mathbf{%
v}_{1}\right) dx+  \notag \\
\mu _{2}\dint_{\Omega _{2}}\left\vert D\left( \mathbf{u}_{2}\right)
\right\vert ^{p_{_{2}}-2}D\left( \mathbf{u}_{2}\right) \cdot D\left( \mathbf{%
v}_{2}\right) dx.  \tag{2.12}  \label{2.12}
\end{gather}

We begin by recalling the following, which gives some properties of the
convective operators $B_{i}.$

\begin{lemma}
\textbf{1.} Suppose that%
\begin{equation}
\frac{3n}{n+2}\leq p_{i}\leq 2,\text{ }i=1,2.  \tag{2.13}  \label{2.13}
\end{equation}

Then,

$B_{i},$ $i=1,2$ is\ trilinear,\ continuous\ on $V\left( \Omega _{i}\right)
\times V\left( \Omega _{i}\right) \times V\left( \Omega _{i}\right) \mathbf{.%
}$ Moreover, $\forall \left( \mathbf{v}_{1},\mathbf{v}_{2},\mathbf{v}%
_{3}\right) \in V\left( \Omega _{i}\right) \times V\left( \Omega _{i}\right)
\times V\left( \Omega _{i}\right) $ we have%
\begin{eqnarray}
B_{i}\left( \mathbf{v}_{1},\mathbf{v}_{2},\mathbf{v}_{3}\right) +B_{i}\left( 
\mathbf{v}_{1},\mathbf{v}_{3},\mathbf{v}_{2}\right) &=&  \notag \\
\left( -1\right) ^{i+1}\dint_{\Gamma _{0}}\left( \mathbf{v}_{1}\cdot \mathbf{%
n}\right) \left( \mathbf{v}_{2}\cdot \mathbf{v}_{3}\right) d\gamma _{0},%
\text{ }i &=&1,2.  \TCItag{2.14}  \label{2.14}
\end{eqnarray}%
where $d\gamma _{0}$\ represents the superficial measure on the boundary
part $\Gamma _{0}.$

\textbf{2.} The operator $\phi $ is hemi-continuous, strictly monotone,
bounded\ and coercive on $V.$
\end{lemma}

\begin{proof}
1.\ The proof of the continuity of $B_{i},$ $i=1,2$\ on $V\left( \Omega
_{i}\right) \times V\left( \Omega _{i}\right) \times V\left( \Omega
_{i}\right) $ is an immediate consequence of H\"{o}lder's inequality and
Sobolev embeddings, see \cite{r9}.

Concerning the equality (\ref{2.13})\ it is enough to use an integration by
parts and using the incompressibility condition (\ref{2.5}), (\ref{2.6}),
the boundary conditions (\ref{2.7}), (\ref{2.8}) and the transmission
condition (\ref{2.9}).

2. We can easily prove that the operator $\phi $ can be written%
\begin{gather*}
\left\langle \phi \left( \mathbf{u}_{1},\mathbf{u}_{2}\right) ,\left( 
\mathbf{v}_{1},\mathbf{v}_{2}\right) \right\rangle _{V^{\prime }\times
V}=\left\langle dJ_{1}\left( D\left( \mathbf{u}_{1}\right) \right) ,D\left( 
\mathbf{v}_{1}\right) \right\rangle _{L^{p_{1}^{\prime }}\left( \Omega
_{1}\right) _{s}^{n\times n}\times L^{p_{1}}\left( \Omega _{1}\right)
_{s}^{n\times n}}+ \\
\left\langle dJ_{2}\left( D\left( \mathbf{u}_{2}\right) \right) ,D\left( 
\mathbf{v}_{2}\right) \right\rangle _{L^{p_{2}^{\prime }}\left( \Omega
_{2}\right) _{s}^{n\times n}\times L^{p_{2}}\left( \Omega _{2}\right)
_{s}^{n\times n}},
\end{gather*}%
where the functional $J_{i},$ $i=1,2$ is defined by%
\begin{equation*}
J_{i}:L^{p_{i}}\left( \Omega _{i}\right) _{s}^{n\times n}\subset \mathbb{S}%
_{n}\longrightarrow 
\mathbb{R}
,\ \ \mathbf{\sigma \longmapsto }J_{i}\left( \sigma \right) =\dfrac{\mu _{i}%
}{p_{i}}\dint_{\Omega _{i}}\left\vert \sigma \right\vert ^{p_{i}}dx,\text{ }%
i=1,2,
\end{equation*}%
and $d$ represents the G\^{a}teaux derivate.

Furthermore, it easy to check that the functional $J_{i},$ $i=1,2$ is convex
and G\^{a}teaux differentiable on $L^{p_{i}}\left( \Omega _{i}\right)
_{s}^{n\times n}.$ Thus, $dJ_{i}$ is hemi-continuous and monotone. The
Gateaux derivate of $J_{i}$ at any point $\sigma \in L^{p_{i}}\left( \Omega
_{i}\right) _{s}^{n\times n}$ is given by%
\begin{gather*}
\left\langle dJ_{i}\left( \sigma \right) ,\tau \right\rangle
_{L^{p_{i}^{\prime }}\left( \Omega _{i}\right) _{s}^{n\times n}\times
L^{p_{i}}\left( \Omega _{i}\right) _{s}^{n\times n}}=\dint_{\Omega _{i}}\mu
_{i}\left\vert \sigma \right\vert ^{p_{i}-2}\sigma \cdot \tau dx\  \\
\forall \tau \in L^{p_{i}}\left( \Omega _{i}\right) _{s}^{n\times n},\text{ }%
i=1,2.
\end{gather*}

This leads after some algebraic manipulations, that for $i=1,2$ 
\begin{eqnarray*}
&&\left\langle dJ_{i}\left( \sigma _{1}\right) -dJ_{i}\left( \sigma
_{2}\right) ,\sigma _{1}-\sigma _{2}\right\rangle _{L^{p_{i}^{\prime
}}\left( \Omega _{i}\right) _{s}^{n\times n}\times L^{p_{i}}\left( \Omega
_{i}\right) _{s}^{n\times n}} \\
&\geq &\mu _{i}\dint_{\Omega _{i}}\left( \left\vert dJ_{i}\left( \sigma
_{1}\right) \right\vert -\left\vert dJ_{i}\left( \sigma _{2}\right)
\right\vert \right) \left( \left\vert \sigma _{1}\right\vert -\left\vert
\sigma _{2}\right\vert \right) dx.
\end{eqnarray*}

Then, if $\sigma _{1}\neq \sigma _{2}$ we find%
\begin{equation*}
\left\langle dJ_{i}\left( \sigma _{1}\right) -dJ_{i}\left( \sigma
_{2}\right) ,\sigma _{1}-\sigma _{2}\right\rangle _{L^{p_{i}^{\prime
}}\left( \Omega _{i}\right) _{s}^{n\times n}\times L^{p_{i}}\left( \Omega
_{i}\right) _{s}^{n\times n}}>0.
\end{equation*}

Which means that the functional $dJ_{i},$ $i=1,2$ is strictly monotone.

Consequently, operator $\phi $ is hemi-continuous and strictly monotone on $%
V.$

Moreover, we obtain from the definition of $\phi $%
\begin{gather*}
\left\vert \left\langle \phi \left( \mathbf{u}_{1},\mathbf{u}_{2}\right)
,\left( \mathbf{v}_{1},\mathbf{v}_{2}\right) \right\rangle _{V^{\prime
}\times V}\right\vert \leq \mu _{1}\left\Vert \mathbf{u}_{1}\right\Vert
_{V\left( \Omega _{1}\right) }^{\frac{p_{1}}{p_{1}^{\prime }}}\left\Vert 
\mathbf{v}_{1}\right\Vert _{V\left( \Omega _{1}\right) }+\mu _{2}\left\Vert 
\mathbf{u}_{2}\right\Vert _{V\left( \Omega _{2}\right) }^{\frac{p_{2}}{%
p_{2}^{\prime }}}\left\Vert \mathbf{v}_{2}\right\Vert _{V\left( \Omega
_{2}\right) } \\
\leq \left( \mu _{1}\left\Vert \mathbf{u}_{1}\right\Vert _{V\left( \Omega
_{1}\right) }^{\frac{p_{1}}{p_{1}^{\prime }}}+\mu _{2}\left\Vert \mathbf{u}%
_{2}\right\Vert _{V\left( \Omega _{2}\right) }^{\frac{p_{2}}{p_{2}^{\prime }}%
}\right) \left\Vert \left( \mathbf{v}_{1},\mathbf{v}_{2}\right) \right\Vert
_{V}\text{ \ }\forall \left( \mathbf{u}_{1},\mathbf{u}_{2}\right) ,\text{ }%
\left( \mathbf{u}_{1},\mathbf{u}_{2}\right) \in V.
\end{gather*}

Then,%
\begin{equation*}
\left\Vert \phi \left( \mathbf{u}_{1},\mathbf{u}_{2}\right) \right\Vert
_{V^{\prime }}\leq \mu _{1}\left\Vert \mathbf{u}_{1}\right\Vert _{V\left(
\Omega _{1}\right) }^{\frac{p_{1}}{p_{1}^{\prime }}}+\mu _{2}\left\Vert 
\mathbf{u}_{2}\right\Vert _{V\left( \Omega _{2}\right) }^{\frac{p_{2}}{%
p_{2}^{\prime }}}\text{ \ }\forall \left( \mathbf{u}_{1},\mathbf{u}%
_{2}\right) \in V.
\end{equation*}

Hence, $\phi $ is bounded on $V.$

Now, we find using the generalized Korn inequality%
\begin{equation*}
\left\langle \phi \left( \mathbf{u}_{1},\mathbf{u}_{2}\right) ,\left( 
\mathbf{u}_{1},\mathbf{u}_{2}\right) \right\rangle _{V^{\prime }\times
V}\geq c\left( \left\Vert \mathbf{u}_{1}\right\Vert _{V\left( \Omega
_{1}\right) }^{p_{1}}+\left\Vert \mathbf{u}_{2}\right\Vert _{V\left( \Omega
_{2}\right) }^{p_{2}}\right) .
\end{equation*}

It follows that%
\begin{equation*}
\frac{\left\langle \phi \left( \mathbf{u}_{1},\mathbf{u}_{2}\right) ,\left( 
\mathbf{u}_{1},\mathbf{u}_{2}\right) \right\rangle _{V^{\prime }\times V}}{%
\left\Vert \left( \mathbf{u}_{1},\mathbf{u}_{2}\right) \right\Vert _{V}}\geq
c\frac{\left\Vert \mathbf{u}_{1}\right\Vert _{V\left( \Omega _{1}\right)
}^{p_{1}}+\left\Vert \mathbf{u}_{2}\right\Vert _{V\left( \Omega _{2}\right)
}^{p_{2}}}{\left\Vert \mathbf{u}_{1}\right\Vert _{V\left( \Omega _{1}\right)
}+\left\Vert \mathbf{u}_{2}\right\Vert _{V\left( \Omega _{2}\right) }}.
\end{equation*}

By passage to the limit when $\left\Vert \left( \mathbf{u}_{1},\mathbf{u}%
_{2}\right) \right\Vert _{V}$ $\longrightarrow +\infty ,$ we find%
\begin{equation*}
\frac{\left\langle \phi \left( \mathbf{u}_{1},\mathbf{u}_{2}\right) ,\left( 
\mathbf{u}_{1},\mathbf{u}_{2}\right) \right\rangle _{V^{\prime }\times V}}{%
\left\Vert \left( \mathbf{u}_{1},\mathbf{u}_{2}\right) \right\Vert _{V}}\geq
c\underset{r\rightarrow +\infty ,\theta \in \left[ 0,\frac{\pi }{2}\right] }{%
\lim }\frac{r^{p_{1}}\cos ^{p_{1}}\theta +r^{p_{2}}\cos ^{p_{2}}\theta }{%
r\cos \theta +r\cos \theta }=+\infty .
\end{equation*}

This proves that the operator $\phi $ is coercive.

Which permits us de conclude the proof.
\end{proof}

\begin{remark}
In (\ref{2.14}), the right hand side has sense, since the injection%
\begin{equation*}
W^{1-\frac{1}{p_{i}},p_{i}}\left( \Gamma _{0}\right) ^{n}\longrightarrow L^{%
\frac{\left( n-1\right) p_{i}}{n-p_{i}}}\left( \Gamma _{0}\right) ^{n},\text{
}i=1,2
\end{equation*}%
is continuous. In particular the trace application%
\begin{equation*}
\gamma _{0}:W^{1,p_{i}}\left( \Omega _{i}\right) ^{n}\longrightarrow
L^{3}\left( \Gamma _{0}\right) ^{n},\text{ }i=1,2
\end{equation*}%
is continuous.
\end{remark}

From now on, we take $\dfrac{3n}{n+2}\leq p_{i}\leq 2,$ $i=1,2.$ The use of
Green's formula under the conditions (2.5)-(2.10) permits us to derive the
following variational formulation of the mechanical problem (P1).

\textit{Problem P}$_{\text{v}}\mathit{.}$\textit{\ }For prescribed data $%
\left( \mathbf{f}_{1},\mathbf{f}_{2}\right) \in V^{\prime }.$ Find $\left( 
\mathbf{u}_{1},\mathbf{u}_{2}\right) \in V$ satisfying the variational
inequality%
\begin{gather}
B_{1}\left( \mathbf{u}_{1},\mathbf{u}_{1},\mathbf{v}_{1}-\mathbf{u}%
_{1}\right) +B_{2}\left( \mathbf{u}_{2},\mathbf{u}_{2},\mathbf{v}_{2}-%
\mathbf{u}_{2}\right) +  \notag \\
\left\langle \phi \left( \mathbf{u}_{1},\mathbf{u}_{2}\right) ,\left( 
\mathbf{v}_{1}-\mathbf{u}_{1},\mathbf{v}_{2}-\mathbf{u}_{2}\right)
\right\rangle _{V^{\prime }\times V}  \notag \\
+g_{1}\dint_{\Omega _{1}}\left\vert D\left( \mathbf{v}_{1}\right)
\right\vert dx-g_{1}\dint_{\Omega _{1}}\left\vert D\left( \mathbf{u}%
_{1}\right) \right\vert dx+g_{2}\dint_{\Omega _{2}}\left\vert D\left( 
\mathbf{v}_{2}\right) \right\vert dx-g_{2}\dint_{\Omega _{2}}\left\vert
D\left( \mathbf{u}_{2}\right) \right\vert dx  \notag \\
\geq \dint_{\Omega _{1}}\mathbf{f}_{1}\cdot \left( \mathbf{v}_{1}-\mathbf{u}%
_{1}\right) dx+\dint_{\Omega _{2}}\mathbf{f}_{2}\cdot \left( \mathbf{v}_{2}-%
\mathbf{u}_{2}\right) dx\ \text{ \ }\forall \left( \mathbf{v}_{1},\mathbf{v}%
_{2}\right) \in V.  \tag{2.15}  \label{2.15}
\end{gather}

\section{Main Result}

In this section we establish an existence result to the problems (\textit{P}$%
_{\text{v}}$).

\begin{theorem}
The problem\ (\textit{P}$_{\text{v}}$)\ admits a solution\ $\left( \mathbf{u}%
_{1},\mathbf{u}_{2}\right) \in V.$
\end{theorem}

The proof will be done in two steps.

\textbf{First step.} Take an arbitrary element $\left( \mathbf{w}_{1},%
\mathbf{w}_{2}\right) \in V$ and consider the auxiliary problem.

\textit{Problem P}$_{\left( \mathbf{w}_{1},\mathbf{w}_{2}\right) }.$ Find $%
\left( \mathbf{u}_{1},\mathbf{u}_{2}\right) =\left( \mathbf{u}_{1}\left( 
\mathbf{w}_{1},\mathbf{w}_{2}\right) ,\mathbf{u}_{2}\left( \mathbf{w}_{1},%
\mathbf{w}_{2}\right) \right) \in V$ solution of the variational inequality%
\begin{gather}
B_{1}\left( \mathbf{w}_{1},\mathbf{u}_{1},\mathbf{v}_{1}-\mathbf{u}%
_{1}\right) +B_{2}\left( \mathbf{w}_{2},\mathbf{u}_{2},\mathbf{v}_{2}-%
\mathbf{u}_{2}\right) +  \notag \\
\left\langle \phi \left( \mathbf{u}_{1},\mathbf{u}_{2}\right) ,\left( 
\mathbf{v}_{1}-\mathbf{u}_{1},\mathbf{v}_{2}-\mathbf{u}_{2}\right)
\right\rangle _{V^{\prime }\times V}  \notag \\
+g_{1}\dint_{\Omega _{1}}\left\vert D\left( \mathbf{v}_{1}\right)
\right\vert dx-g_{1}\dint_{\Omega _{1}}\left\vert D\left( \mathbf{u}%
_{1}\right) \right\vert dx+g_{2}\dint_{\Omega _{2}}\left\vert D\left( 
\mathbf{v}_{2}\right) \right\vert dx-g_{2}\dint_{\Omega _{2}}\left\vert
D\left( \mathbf{u}_{2}\right) \right\vert dx  \notag \\
\geq \dint_{\Omega _{1}}\mathbf{f}_{1}\cdot \left( \mathbf{v}_{1}-\mathbf{u}%
_{1}\right) dx+\dint_{\Omega _{2}}\mathbf{f}_{2}\cdot \left( \mathbf{v}_{2}-%
\mathbf{u}_{2}\right) dx\ \text{ \ }\forall \left( \mathbf{v}_{1},\mathbf{v}%
_{2}\right) \in V.  \tag{3.1}  \label{3.1}
\end{gather}%
\noindent and it satisfies the estimate%
\begin{equation}
\left\Vert \left( \mathbf{u}_{1},\mathbf{u}_{2}\right) \right\Vert _{V}\leq
R.  \tag{3.2}  \label{3.2}
\end{equation}

\begin{lemma}
The problem (\textit{P}$_{\left( \mathbf{w}_{1},\mathbf{w}_{2}\right) }$)
has a unique solution%
\begin{equation*}
\left( \mathbf{u}_{1},\mathbf{u}_{2}\right) =\left( \mathbf{u}_{1}\left( 
\mathbf{w}_{1},\mathbf{w}_{2}\right) ,\mathbf{u}_{2}\left( \mathbf{w}_{1},%
\mathbf{w}_{2}\right) \right) \in V.
\end{equation*}
\end{lemma}

\begin{proof}
Let us introduce the operator%
\begin{gather}
\phi _{\left( \mathbf{w}_{1},\mathbf{w}_{2}\right) }:V\longrightarrow
V^{\prime },\text{ }\left( \mathbf{u}_{1},\mathbf{u}_{2}\right) \longmapsto
\phi _{\left( \mathbf{w}_{1},\mathbf{w}_{2}\right) }\left( \mathbf{u}_{1},%
\mathbf{u}_{2}\right) :\forall \left( \mathbf{v}_{1},\mathbf{v}_{2}\right)
\in V  \notag \\
\left\langle \phi _{\left( \mathbf{w}_{1},\mathbf{w}_{2}\right) }\left( 
\mathbf{u}_{1},\mathbf{u}_{2}\right) ,\left( \mathbf{v}_{1},\mathbf{v}%
_{2}\right) \right\rangle _{V^{\prime }\times V}=B_{1}\left( \mathbf{w}_{1},%
\mathbf{u}_{1},\mathbf{v}_{1}\right) +B_{2}\left( \mathbf{w}_{2},\mathbf{u}%
_{2},\mathbf{v}_{2}\right) +  \notag \\
\left\langle \phi \left( \mathbf{u}_{1},\mathbf{u}_{2}\right) ,\left( 
\mathbf{v}_{1},\mathbf{v}_{2}\right) \right\rangle _{V^{\prime }\times V}. 
\tag{3.3}  \label{3.3}
\end{gather}

First, we get using lemma 2.1%
\begin{gather*}
B_{1}\left( \mathbf{w}_{1},\mathbf{u}_{1},\mathbf{u}_{1}\right) +B_{2}\left( 
\mathbf{w}_{2},\mathbf{u}_{2},\mathbf{u}_{2}\right) =\dfrac{1}{2}%
\dint_{\Gamma _{0}}\left\vert \mathbf{u}_{1}\right\vert ^{2}\left( \mathbf{w}%
_{1}\cdot \mathbf{n}\right) d\gamma _{0}- \\
\dfrac{1}{2}\dint_{\Gamma _{0}}\left\vert \mathbf{u}_{2}\right\vert
^{2}\left( \mathbf{w}_{2}\cdot \mathbf{n}\right) d\gamma _{0}\ \ \forall
\left( \mathbf{u}_{1},\mathbf{u}_{2}\right) \in V.
\end{gather*}

The fact that $\left( \mathbf{w}_{1},\mathbf{w}_{2}\right) ,$ $\left( 
\mathbf{u}_{1},\mathbf{u}_{2}\right) \in V$ implies that $\mathbf{w}_{1}-%
\mathbf{w}_{2}=0$ and $\mathbf{u}_{1}-\mathbf{u}_{2}=0$ on $\Gamma _{0}.$
Thus, 
\begin{eqnarray}
\left\langle \phi _{\left( \mathbf{w}_{1},\mathbf{w}_{2}\right) }\left( 
\mathbf{u}_{1},\mathbf{u}_{2}\right) ,\left( \mathbf{u}_{1},\mathbf{u}%
_{2}\right) \right\rangle _{V^{\prime }\times V} &=&\left\langle \phi \left( 
\mathbf{u}_{1},\mathbf{u}_{2}\right) ,\left( \mathbf{u}_{1},\mathbf{u}%
_{2}\right) \right\rangle _{V^{\prime }\times V}\text{ }  \notag \\
\forall \left( \mathbf{u}_{1},\mathbf{u}_{2}\right) &\in &V.  \TCItag{3.4}
\label{3.4}
\end{eqnarray}

Furthermore, we find for every $\left( \mathbf{u}_{1},\mathbf{u}_{2}\right)
, $ $\left( \mathbf{v}_{1},\mathbf{v}_{2}\right) \in V.$ 
\begin{gather*}
\left\langle \phi _{\left( \mathbf{w}_{1},\mathbf{w}_{2}\right) }\left( 
\mathbf{v}_{1},\mathbf{v}_{2}\right) -\phi _{\left( \mathbf{w}_{1},\mathbf{w}%
_{2}\right) }\left( \mathbf{u}_{1},\mathbf{u}_{2}\right) ,\left( \mathbf{v}%
_{1}-\mathbf{u}_{1},\mathbf{v}_{2}-\mathbf{u}_{2}\right) \right\rangle
_{V^{\prime }\times V} \\
=B_{1}\left( \mathbf{w}_{1},\mathbf{u}_{1},\mathbf{v}_{1}-\mathbf{u}%
_{1}\right) +B_{2}\left( \mathbf{w}_{2},\mathbf{u}_{2},\mathbf{v}_{2}-%
\mathbf{u}_{2}\right) -B_{1}\left( \mathbf{w}_{1},\mathbf{v}_{1},\mathbf{v}%
_{1}-\mathbf{u}_{1}\right) - \\
B_{2}\left( \mathbf{w}_{2},\mathbf{v}_{2},\mathbf{v}_{2}-\mathbf{u}%
_{2}\right) +\left\langle \phi \left( \mathbf{v}_{1},\mathbf{v}_{2}\right)
-\phi \left( \mathbf{u}_{1},\mathbf{u}_{2}\right) ,\left( \mathbf{v}_{1},%
\mathbf{v}_{2}\right) -\left( \mathbf{u}_{1},\mathbf{u}_{2}\right)
\right\rangle _{V^{\prime }\times V} \\
\forall \left( \mathbf{v}_{1},\mathbf{v}_{2}\right) ,\text{ }\left( \mathbf{u%
}_{1},\mathbf{u}_{2}\right) \in V.
\end{gather*}

This gives, keeping in mind lemma 2.1%
\begin{gather}
\left\langle \phi _{\left( \mathbf{w}_{1},\mathbf{w}_{2}\right) }\left( 
\mathbf{v}_{1},\mathbf{v}_{2}\right) -\phi _{\left( \mathbf{w}_{1},\mathbf{w}%
_{2}\right) }\left( \mathbf{u}_{1},\mathbf{u}_{2}\right) ,\left( \mathbf{v}%
_{1}-\mathbf{u}_{1},\mathbf{v}_{2}-\mathbf{u}_{2}\right) \right\rangle
_{V^{\prime }\times V}  \notag \\
=\dfrac{1}{2}\dint_{\Gamma _{0}}\left\vert \mathbf{u}_{1}-\mathbf{v}%
_{1}\right\vert ^{2}\left( \mathbf{w}_{2}\cdot \mathbf{n}\right) d\gamma
_{0}-\dfrac{1}{2}\dint_{\Gamma _{0}}\left\vert \mathbf{u}_{2}-\mathbf{v}%
_{2}\right\vert ^{2}\left( \mathbf{w}_{1}\cdot \mathbf{n}\right) d\gamma _{0}
\notag \\
+\left\langle \phi \left( \mathbf{v}_{1},\mathbf{v}_{2}\right) -\phi \left( 
\mathbf{u}_{1},\mathbf{u}_{2}\right) ,\left( \mathbf{v}_{1}-\mathbf{u}_{1},%
\mathbf{v}_{2}-\mathbf{u}_{2}\right) \right\rangle _{V^{\prime }\times V} 
\notag \\
=\left\langle \phi \left( \mathbf{v}_{1},\mathbf{v}_{2}\right) -\phi \left( 
\mathbf{u}_{1},\mathbf{u}_{2}\right) ,\left( \mathbf{v}_{1}-\mathbf{u}_{1},%
\mathbf{v}_{2}-\mathbf{u}_{2}\right) \right\rangle _{V^{\prime }\times V}%
\text{ }  \notag \\
\forall \left( \mathbf{v}_{1},\mathbf{v}_{2}\right) ,\text{ }\left( \mathbf{u%
}_{1},\mathbf{u}_{2}\right) \in V.  \tag{3.5}  \label{3.5}
\end{gather}

Consequently, lemma 2.1 leads making use (\ref{3.4}) and (\ref{3.5}) that
the operator is $\phi _{\left( \mathbf{w}_{1},\mathbf{w}_{2}\right) }$ is
hemi-continuous, strictly monotone, bounded\ and coercive on $V$ for every $%
\left( \mathbf{w}_{1},\mathbf{w}_{2}\right) \in V.$

Consider now the following functional%
\begin{gather}
j:V\longrightarrow 
\mathbb{R}
,  \notag \\
j\left( \mathbf{v}_{1},\mathbf{v}_{2}\right) =g_{1}\dint_{\Omega
_{1}}\left\vert D\left( \mathbf{v}_{1}\right) \right\vert
dx+g_{2}\dint_{\Omega _{2}}\left\vert D\left( \mathbf{v}_{2}\right)
\right\vert dx  \tag{3.6}  \label{3.6}
\end{gather}

We can easily verify that the functional $j$ is proper, convex and lower
semi-continuous on $V.$

The inequality (\ref{3.1}) can be rewritten using the operator $\phi
_{\left( \mathbf{w}_{1},\mathbf{w}_{2}\right) }$ and the functional $j$ as
follows%
\begin{gather}
\left\langle \phi _{\left( \mathbf{w}_{1},\mathbf{w}_{2}\right) }\left( 
\mathbf{u}_{1},\mathbf{u}_{2}\right) ,\left( \mathbf{v}_{1}-\mathbf{u}_{1},%
\mathbf{v}_{2}-\mathbf{u}_{2}\right) \right\rangle _{V^{\prime }\times
V}+j\left( \mathbf{v}_{1},\mathbf{v}_{2}\right) -j\left( \mathbf{u}_{1},%
\mathbf{u}_{2}\right)  \notag \\
\geq \dint_{\Omega _{1}}\mathbf{f}_{1}\cdot \left( \mathbf{v}_{1}-\mathbf{u}%
_{1}\right) dx+\dint_{\Omega _{2}}\mathbf{f}_{2}\cdot \left( \mathbf{v}_{2}-%
\mathbf{u}_{2}\right) dx\ \ \forall \left( \mathbf{v}_{1},\mathbf{v}%
_{2}\right) \in V.  \tag{3.7}  \label{3.7}
\end{gather}

Consequently, the existence and uniqueness results from classical theories
for inequalities with monotone operators and convex functionals, see \cite%
{r1}.

Furthermore the estimate (\ref{3.2}) can be easily deduced by setting $%
\left( \mathbf{v}_{1},\mathbf{v}_{2}\right) =\left( 0,0\right) $ as test
function in inequality (\ref{3.1}), using lemma 2.1, Korn's inequality and
some algebraic manipulations.
\end{proof}

\textbf{Second step. }In order to obtain the solution of problem (\textit{P}$%
_{\text{v}}$) from that of problem \textit{P}$_{\left( \mathbf{w}_{1},%
\mathbf{w}_{2}\right) },$ we use the Schauder fixed point theorem, see \cite%
{r5}. To this aim we introduce the ball%
\begin{equation}
K=\left\{ \left( \mathbf{w}_{1},\mathbf{w}_{2}\right) \in V:\left\Vert
\left( \mathbf{w}_{1},\mathbf{w}_{2}\right) \right\Vert _{V}\leq R\right\} ,
\tag{3.8}  \label{3.8}
\end{equation}%
where $R$ is the constant given by the estimate (\ref{3.2}). The ball $K$ is
convex and from the Rellich compactness theorem the ball is compact in $L^{%
\frac{3n}{n-1}}\left( \Omega _{1}\right) ^{n}\times L^{\frac{3n}{n-1}}\left(
\Omega _{2}\right) ^{n}.$ Let us built the mapping $\mathbf{\tciLaplace }%
:K\longrightarrow K,$ as follows%
\begin{equation*}
\left( \mathbf{w}_{1},\mathbf{w}_{2}\right) \longmapsto \mathbf{\tciLaplace }%
\left( \mathbf{w}_{1},\mathbf{w}_{2}\right) =\left( \mathbf{u}_{1},\mathbf{u}%
_{2}\right) \mathbf{.}
\end{equation*}

To conclude the proof it is enough to verify the continuity of the mapping $%
\mathbf{\tciLaplace }$ when the ball $K$ is provided by the topology of
space $L^{\frac{3n}{n-1}}\left( \Omega _{1}\right) ^{n}\times L^{\frac{3n}{%
n-1}}\left( \Omega _{2}\right) ^{n}.$ To do this, we consider $\left( 
\mathbf{w}_{1},\mathbf{w}_{2}\right) ,$ $\left( \mathbf{w}_{1}^{\prime },%
\mathbf{w}_{2}^{\prime }\right) \in K$\ and denoting by $\left( \mathbf{u}%
_{1},\mathbf{u}_{2}\right) ,$ $\left( \mathbf{u}_{1}^{\prime },\mathbf{u}%
_{2}^{\prime }\right) \in K$ the elements $\left( \mathbf{u}_{1},\mathbf{u}%
_{2}\right) =\mathbf{\tciLaplace }\left( \mathbf{w}_{1},\mathbf{w}%
_{2}\right) \ $and $\left( \mathbf{u}_{1}^{\prime },\mathbf{u}_{2}^{\prime
}\right) =\mathbf{\tciLaplace }\left( \mathbf{w}_{1}^{\prime },\mathbf{w}%
_{2}^{\prime }\right) .$

Remembering that $\left( \mathbf{u}_{1},\mathbf{u}_{2}\right) $ and $\left( 
\mathbf{u}_{1}^{\prime },\mathbf{u}_{2}^{\prime }\right) $ are the solution
of the problems below%
\begin{gather}
B_{1}\left( \mathbf{w}_{1},\mathbf{u}_{1},\mathbf{v}_{1}-\mathbf{u}%
_{1}\right) +B_{2}\left( \mathbf{w}_{2},\mathbf{u}_{2},\mathbf{v}_{2}-%
\mathbf{u}_{2}\right) +  \notag \\
\left\langle \phi \left( \mathbf{u}_{1},\mathbf{u}_{2}\right) ,\left( 
\mathbf{v}_{1}-\mathbf{u}_{1},\mathbf{v}_{2}-\mathbf{u}_{2}\right)
\right\rangle _{V^{\prime }\times V}  \notag \\
+g_{1}\dint_{\Omega _{1}}\left\vert D\left( \mathbf{v}_{1}\right)
\right\vert dx-g_{1}\dint_{\Omega _{1}}\left\vert D\left( \mathbf{u}%
_{1}\right) \right\vert dx+g_{2}\dint_{\Omega _{2}}\left\vert D\left( 
\mathbf{v}_{2}\right) \right\vert dx-g_{2}\dint_{\Omega _{2}}\left\vert
D\left( \mathbf{u}_{2}\right) \right\vert dx  \notag \\
\geq \dint_{\Omega _{1}}\mathbf{f}_{1}\cdot \left( \mathbf{v}_{1}-\mathbf{u}%
_{1}\right) dx+\dint_{\Omega _{2}}\mathbf{f}_{2}\cdot \left( \mathbf{v}_{2}-%
\mathbf{u}_{2}\right) dx\ \text{ \ }\forall \left( \mathbf{v}_{1},\mathbf{v}%
_{2}\right) \in V,  \tag{3.9}  \label{3.9}
\end{gather}%
and%
\begin{gather}
B_{1}\left( \mathbf{w}_{1}^{\prime },\mathbf{u}_{1}^{\prime },\mathbf{v}_{1}-%
\mathbf{u}_{1}^{\prime }\right) +B_{2}\left( \mathbf{w}_{2}^{\prime },%
\mathbf{u}_{2}^{\prime },\mathbf{v}_{2}-\mathbf{u}_{2}^{\prime }\right) + 
\notag \\
\left\langle \phi \left( \mathbf{u}_{1}^{\prime },\mathbf{u}_{2}^{\prime
}\right) ,\left( \mathbf{v}_{1}-\mathbf{u}_{1}^{\prime },\mathbf{v}_{2}-%
\mathbf{u}_{2}^{\prime }\right) \right\rangle _{V^{\prime }\times V}  \notag
\\
+g_{1}\dint_{\Omega _{1}}\left\vert D\left( \mathbf{v}_{1}\right)
\right\vert dx-g_{1}\dint_{\Omega _{1}}\left\vert D\left( \mathbf{u}%
_{1}^{\prime }\right) \right\vert dx+g_{2}\dint_{\Omega _{2}}\left\vert
D\left( \mathbf{v}_{2}\right) \right\vert dx-g_{2}\dint_{\Omega
_{2}}\left\vert D\left( \mathbf{u}_{2}^{\prime }\right) \right\vert dx 
\notag \\
\geq \dint_{\Omega _{1}}\mathbf{f}_{1}\cdot \left( \mathbf{v}_{1}-\mathbf{u}%
_{1}^{\prime }\right) dx+\dint_{\Omega _{2}}\mathbf{f}_{2}\cdot \left( 
\mathbf{v}_{2}-\mathbf{u}_{2}^{\prime }\right) dx\ \text{ \ }\forall \left( 
\mathbf{v}_{1},\mathbf{v}_{2}\right) \in V,  \tag{3.10}  \label{3.10}
\end{gather}

Now, choosing $\mathbf{v}_{1}=\mathbf{u}_{1}^{\prime }$ and $\mathbf{v}_{2}=%
\mathbf{u}_{2}^{\prime }$ as test function in inequality (\ref{3.9}) and $%
\mathbf{v}_{1}=\mathbf{u}_{1}$ and $\mathbf{v}_{2}=\mathbf{u}_{2}$ as test
function in inequality (\ref{3.10}). It follows by subtracting the two
obtained inequalities and using lemma 3.1, the transmission conditions, the
definition of the space $V$ and some calculations%
\begin{gather}
B_{1}\left( \mathbf{w}_{1}^{\prime }-\mathbf{w}_{1},\mathbf{u}_{1},\mathbf{u}%
_{1}^{\prime }-\mathbf{u}_{1}\right) +B_{2}\left( \mathbf{w}_{2}^{\prime }-%
\mathbf{w}_{2},\mathbf{u}_{2},\mathbf{u}_{2}^{\prime }-\mathbf{u}_{2}\right)
+  \notag \\
\dfrac{1}{2}\dint_{\Gamma _{0}}\left\vert \mathbf{u}_{1}^{\prime }-\mathbf{u}%
_{1}\right\vert ^{2}\left( \mathbf{w}_{1}^{\prime }\cdot \mathbf{n}\right)
d\gamma _{0}-\dfrac{1}{2}\dint_{\Gamma _{0}}\left\vert \mathbf{u}%
_{2}^{\prime }-\mathbf{u}_{2}\right\vert ^{2}\left( \mathbf{w}_{2}^{\prime
}\cdot \mathbf{n}\right) d\gamma _{0}  \notag \\
+\mu _{1}\dint_{\Omega _{1}}\left( \left\vert D\left( \mathbf{u}_{1}^{\prime
}\right) \right\vert ^{p_{_{1}}-2}D\left( \mathbf{u}_{1}^{\prime }\right)
-\left\vert D\left( \mathbf{u}_{1}\right) \right\vert ^{p_{_{1}}-2}D\left( 
\mathbf{u}_{1}\right) \right) \cdot D\left( \mathbf{u}_{1}^{\prime }-\mathbf{%
u}_{1}\right) dx  \notag \\
+\mu _{2}\dint_{\Omega _{2}}\left( \left\vert D\left( \mathbf{u}_{2}^{\prime
}\right) \right\vert ^{p_{2}-2}D\left( \mathbf{u}_{2}^{\prime }\right)
-\left\vert D\left( \mathbf{u}_{2}\right) \right\vert ^{p_{_{2}}-2}D\left( 
\mathbf{u}_{2}\right) \right) \cdot D\left( \mathbf{u}_{2}^{\prime }-\mathbf{%
u}_{2}\right) dx\leq 0.  \tag{3.11}  \label{3.11}
\end{gather}

Observe that for every $x,y\in 
\mathbb{R}
^{n},$%
\begin{equation}
\left( \left\vert x\right\vert ^{p-2}x-\left\vert y\right\vert
^{p-2}y\right) \cdot \left( x-y\right) \geq c\dfrac{\left\vert
x-y\right\vert ^{2}}{\left( \left\vert x\right\vert +\left\vert y\right\vert
\right) ^{2-p}},\text{ \ }1<p\leq 2.  \tag{3.12}  \label{3.12}
\end{equation}

Then, inequality (\ref{3.11}) becomes%
\begin{gather}
\mu _{1}\dint_{\Omega _{1}}\dfrac{\left\vert D\left( \mathbf{u}_{1}^{\prime
}-\mathbf{u}_{1}\right) \right\vert ^{2}}{\left( \left\vert D\left( \mathbf{u%
}_{1}\right) \right\vert +\left\vert D\left( \mathbf{u}_{1}^{\prime }\right)
\right\vert \right) ^{2-p_{_{1}}}}dx+\mu _{2}\dint_{\Omega _{2}}\dfrac{%
\left\vert D\left( \mathbf{u}_{2}^{\prime }-\mathbf{u}_{2}\right)
\right\vert ^{2}}{\left( \left\vert D\left( \mathbf{u}_{2}\right)
\right\vert +\left\vert D\left( \mathbf{u}_{2}^{\prime }\right) \right\vert
\right) ^{2-p_{_{2}}}}dx  \notag \\
\leq c\left\vert B_{1}\left( \mathbf{w}_{1}^{\prime }-\mathbf{w}_{1},\mathbf{%
u}_{1},\mathbf{u}_{1}^{\prime }-\mathbf{u}_{1}\right) \right\vert
+c\left\vert B_{2}\left( \mathbf{w}_{2}^{\prime }-\mathbf{w}_{2},\mathbf{u}%
_{2},\mathbf{u}_{2}^{\prime }-\mathbf{u}_{2}\right) \right\vert  \tag{3.13}
\label{3.13}
\end{gather}

On the other hand, the application of Korn's and H\"{o}lder's inequalities
leads for $i=1,2$ to%
\begin{gather}
\left\Vert \mathbf{u}_{i}^{\prime }-\mathbf{u}_{i}\right\Vert _{V\left(
\Omega _{i}\right) }^{p_{i}}  \notag \\
\leq c\left( \dint_{\Omega _{i}}\dfrac{\left\vert D\left( \mathbf{u}%
_{i}^{\prime }-\mathbf{u}_{i}\right) \right\vert ^{2}}{\left( \left\vert
D\left( \mathbf{u}_{i}\right) \right\vert +\left\vert D\left( \mathbf{u}%
_{i}^{\prime }\right) \right\vert \right) ^{2-p_{i}}}dx\right) ^{\frac{p_{i}%
}{2}}\left( \dint_{\Omega _{i}}\left( \left\vert D\left( \mathbf{u}%
_{i}^{\prime }\right) \right\vert +\left\vert D\left( \mathbf{u}_{i}\right)
\right\vert \right) ^{p_{i}}dx\right) ^{\frac{2-p_{i}}{2}}\text{ .} 
\tag{3.14}  \label{3.14}
\end{gather}

This yields, taking into account (\ref{3.2}), (\ref{3.13}) and H\"{o}lder's
inequality%
\begin{gather*}
\left\Vert \left( \mathbf{u}_{1}^{\prime }-\mathbf{u}_{1},\mathbf{u}%
_{2}^{\prime }-\mathbf{u}_{2}\right) \right\Vert _{V}^{2}\leq \\
c\left\Vert \mathbf{w}_{1}^{\prime }-\mathbf{w}_{1}\right\Vert _{L^{\frac{3n%
}{n-1}}\left( \Omega _{1}\right) ^{n}}\left\Vert \mathbf{u}_{1}\right\Vert
_{L^{p_{1}}\left( \Omega _{1}\right) ^{n}}\left\Vert \mathbf{u}_{1}^{\prime
}-\mathbf{u}_{1}\right\Vert _{L^{\frac{3n}{n-1}}\left( \Omega _{1}\right)
^{n}}+ \\
c\left\Vert \mathbf{w}_{2}^{\prime }-\mathbf{w}_{2}\right\Vert _{L^{\frac{3n%
}{n-1}}\left( \Omega _{2}\right) ^{n}}\left\Vert \mathbf{u}_{2}\right\Vert
_{L^{p_{2}}\left( \Omega _{2}\right) ^{n}}\left\Vert \mathbf{u}_{2}^{\prime
}-\mathbf{u}_{2}\right\Vert _{L^{\frac{3n}{n-1}}\left( \Omega _{2}\right)
^{n}}.
\end{gather*}

Thus, Sobolev' embedding leads via the estimate (\ref{3.2}) to%
\begin{gather}
\left\Vert \left( \mathbf{u}_{1}^{\prime }-\mathbf{u}_{1},\mathbf{u}%
_{2}^{\prime }-\mathbf{u}_{2}\right) \right\Vert _{L^{\frac{3n}{n-1}}\left(
\Omega _{1}\right) ^{n}\times L^{\frac{3n}{n-1}}\left( \Omega _{2}\right)
^{n}}\leq  \notag \\
c\left\Vert \left( \mathbf{w}_{1}^{\prime }-\mathbf{w}_{1},\mathbf{w}%
_{2}^{\prime }-\mathbf{w}_{2}\right) \right\Vert _{L^{\frac{3n}{n-1}}\left(
\Omega _{1}\right) ^{n}\times L^{\frac{3n}{n-1}}\left( \Omega _{2}\right)
^{n}}.  \tag{3.15}  \label{3.15}
\end{gather}

Hence, by virtue of Schauder's fixed point theorem, the mapping $\mathbf{%
\tciLaplace }$ admits a fixed point $\left( \mathbf{u}_{1},\mathbf{u}%
_{2}\right) =\mathbf{\tciLaplace }\left( \mathbf{u}_{1},\mathbf{u}%
_{2}\right) ,$ which solves the problem (\textit{P}$_{\text{v}}$).

\end{document}